\documentclass[12pt, reqno]{amsart}
\usepackage{amsmath}
\usepackage{amscd}
\usepackage{bbm}
\usepackage{dsfont}
\usepackage{enumerate}
\usepackage{latexsym}
\usepackage{srcltx}
\usepackage{amsfonts}
\usepackage{amssymb}
\usepackage{datetime}
\usepackage{setspace}
\usepackage{enumerate}
\usepackage{amsthm}
\usepackage{graphicx}
\usepackage{epstopdf}

\newtheorem{theorem}{Theorem}[section]
\newtheorem{proposition}[theorem]{Proposition}
\newtheorem{lemma}[theorem]{Lemma}
\newtheorem{corollary}[theorem]{Corollary}

\theoremstyle{definition}
\newtheorem{example}[theorem]{Example}

\newtheorem{remark} [theorem] {Remark}

\newtheorem{problem} [theorem] {Problem}

\begin{document}
\title{Spectrum of weighted composition operators. Part XII. Kamowitz - Scheinberg theorem revisited}

\author{Arkady Kitover}

\address{Community College of Philadelphia, 1700 Spring Garden St., Philadelphia, PA, USA}

\email{akitover@ccp.edu}

\author{Mehmet Orhon}

\address{University of New Hampshire, 105 Main Street
Durham, NH 03824}

\email{mo@unh.edu}


\date{\today}


\begin{abstract}
 The  well-known Kamowitz - Scheinberg theorem states that if $U$ is an automorphism of a commutative semi-simple Banach algebra and $U^n \neq I, n \in \mathds{N}$, then the spectrum of $U$ contains the unit circle. In this paper we present some results about  the spectrum of weighted automorphisms of           unital commutative semi-simple Banach algebras that considerably strengthen the statement of the Kamowitz - Scheinberg theorem.  
\end{abstract}

\maketitle

\markboth{Arkady Kitover and Mehmet Orhon}{Kamowitz - Scheinberg theorem revisited.}

\section{introduction} Let $A$ be a commutative semi-simple algebra \footnote{All linear spaces are considered over the field of complex numbers $\mathds{C}$.} and $U$ be an automorphism of $A$. Let $\sigma(U)$ be the spectrum of $U$. The Kamowitz - Scheinberg theorem (see~\cite[Theorem 3]{KS}) states the following

\begin{theorem} \label{t1}
  Either $U^N = I$ for some integer $N$ in which case $\sigma(U)$ consists of a finite union of finite subgroups of the circle, or else $\sigma(U)$ contains the entire unit circle $\mathds{T}$. 
\end{theorem}

After the appearance of the paper~\cite{KS} alternative proofs of Kamowitz - Scheinberg theorem appeared in~\cite{Jo} and~\cite{Hu}.  

Unless stated otherwise, $A$ will mean a unital commutative semi-simple Banach algebra, $\mathfrak{M}_A$ and $\partial A$ denote the space of maximal ideals and the Shilov boundary of $A$, respectively.

Let $U$ be an automorphism of $A$. Let $m \in \mathfrak{M}_A$ and let $F_m$ be the corresponding multiplicative functional on $A$. The map $F_m \rightarrow U^\prime F_m$, where $U^\prime$ is the Banach adjoint of $U$ defines the unique homeomorphism $\varphi$ of $\mathfrak{M}_A$ onto itself, and it is immediate to see that $\varphi(\partial A) = \partial A$.

Let $f \in A$ and let $\hat{f} \in C(\mathfrak{M}_A)$ be the Gelfand image of $f$. We denote the restriction of $\hat{f}$ on $\partial A$ by $\check{f}$ and the closure of the set $\{\check{f} : f \in A\}$ in $C(\partial A)$ by $\check{A}$.

Let $w \in A$ and $T = wU$. We associate with the operator $T$ the weighted composition operator $\check{T}$ on $\check{A}$ defined as
\begin{equation*}
  (\check{T}g)(t) = \check{w}(t)g(\varphi(t)), g \in \check{A}, t \in \partial A.
\end{equation*}

\section{main results}

\begin{theorem} \label{t2}
  Let $A$ be a unital semi-simple commutative Banach algebra, $U$ be an automorphism of $A$ and $w \in A^{-1}$. Let
  $T = wU$ and let $\check{T}$ be the associated weighted composition operator on $\check{A}$. Assume that
  
  \begin{equation} \label{eq1}
     \; U^n \neq I, n \in \mathds{N},
  \end{equation}
   
  \begin{equation}\label{eq2}
     \; \sum \limits_{n=-\infty}^\infty \frac{|\ln{\|(\check{T})^n\|}|}{1+n^2} < \infty.
  \end{equation}
  Then $\sigma(T) \supseteq \mathds{T}$.
  \end{theorem}
  
  \begin{proof} Assume to the contrary that there is an $\alpha \in \mathds{T}$ such that $\alpha \not \in \sigma(T)$. Let $U(\alpha, r)$ be an open disc centered at $\alpha$ of a radius $r >0$ such that 
  $\sigma(T) \cap U(\alpha, r) = \emptyset$.
  
  Condition (1) implies (see~\cite{We}) that $\sigma(\check{T}) \subseteq \mathds{T}$.  We claim that 
\begin{equation}\label{eq4}
  ((\lambda I - T)^{-1}f\check{)} = (\lambda I - \check{T})^{-1}\check{f},\; f \in A, \; \lambda \in D(\alpha,r) \setminus \mathds{T}.
\end{equation}
Notice first that for any $g \in A$ and $t \in \partial A$
\begin{equation}\label{eq5}
  (\lambda I - \check{T})\check{g}(t) = ((\lambda I - T)g\check{)}(t).
\end{equation}
Let $t \in \partial A$. Then $((\lambda I -\check{T})(\lambda I - \check{T})^{-1}\check{f})(t) = \check{f}(t)$
and in virtue of~(\ref{eq5}) $(\lambda I - \check{T}) ((\lambda I - T)^{-1}f\check{)}(t) = ((\lambda I - T)(\lambda I - T)^{-1}f\check{)}(t) = \check{f}(t)$. Recall that the operator $\lambda I - \check{T}$ is invertible, and therefore~(\ref{eq4}) is proved.

Let $\check{G}$ be a continuous functional on $\check{A}$. The formula $G(f) = \check{G}(\check{f})$ defines a continuous functional $G$ on $A$. Therefore, in virtue of~(\ref{eq4}), for any $f \in A$ and any $\check{G} \in (\check{A})^\prime$ the function $H(\lambda) = \check{G}((\lambda I - \check{T})^{-1}\check{f})$ is analytic in
$U(\alpha,r)$.

Let $g \in \check{A}$ and let $f_n \in A$ be such that  $\|\check{f_n} - g\|_{\check{A}} \rightarrow 0$. It follows from condition~(\ref{eq1}) (see \cite[Lemma 3]{We}) that the family

 $H_n(\lambda) = \check{G}((\lambda I - \check{T})^{-1}\check{f_n})$ is normal in $U(\alpha, r)$. Therefore the function $\check{G}((\lambda I - \check{T})g)$ is analytic in $U(\alpha, r)$.

It follows (~\cite[theorem 7b]{Du}, see also~\cite[p. 334]{Ta}) that the operator function $R(\lambda,\check{T})$ is analytic in $U(\alpha, r)$ and therefore $\alpha \not \in \sigma(\check{T})$. But it follows from condition~(\ref{eq2}) and from theorems 3.12 and 4.2 in~\cite{Ki} that $\sigma(\check{T}) = \mathds{T}$, a contradiction.
  \end{proof}

The next corollary directly strengthens the Kamowitz - Scheinberg theorem.

\begin{corollary} \label{c1}
  Let $A$ be a commutative semi-simple Banach algebra and $U$ be an automorphism of $A$. Let $B = A \oplus \mathds{1}$ be the corresponding unital semi-simple Banach algebra. We extend $U$ on $B$ in the standard way:
$U\mathds{1} = \mathds{1}$. Let $w \in B$ and $T = wU$. Let $S = T|A$. Assume conditions~(\ref{eq1}) 
and~(\ref{eq2}) of Theorem~\ref{t1}. Then $\sigma(S) \supseteq \mathds{T}$. 
\end{corollary}
\begin{proof} The factor operator $\dot{T}$ acts on the one dimensional factor space $B/A$ and therefore its spectrum is a single point, say $\beta$. If there is an $\alpha \in \mathds{T}$ such that $\alpha \neq \beta$ and $\alpha \not \in \sigma(S)$, then by the three spaces property $\alpha \not \in \sigma(T)$ in contradiction with Theorem~\ref{t1}. 
\end{proof}

In connection with Theorem~\ref{t2} we pose the following problem.

\begin{problem} \label{pr1}
 Let $A$ be a unital semi-simple commutative Banach algebra, $U$ be an automorphism of $A$ and $w \in A$. Let
  $T = wU$ and let $\check{T}$ be the associated weighted composition operator on $\check{A}$. Assume that
$U^n \neq I, n \in \mathds{N}$, and that $\sigma(\check{T}) \subseteq \mathds{T}$. Is it true that 
$\mathds{T} \subseteq \sigma(T)$? 
\end{problem}

The next theorem provides a partial answer to Problem~\ref{pr1} under very heavy restrictions.

\begin{theorem} \label{t3}
  Let $A$ be a unital semi-simple commutative Banach algebra, $U$ be an automorphism of $A$ and $w \in A$. Let
  $T = wU$ and let $\check{T}$ be the associated weighted composition operator on $\check{A}$. Assume that
$U^n \neq I, n \in \mathds{N}$ and that $\sigma(\check{T}) \subseteq \mathds{T}$. Assume additionally that
$\partial A$ contains no $\varphi$-periodic points and that every point of $\partial A$ is a peak point for the algebra $\check{A}$. Then $\mathds{T} \subseteq \sigma(T)$.
\end{theorem}

\begin{proof}
 Consider the operator $S$ on $\check{A}$ defined by the formula
  $(Sg)(t)=\check{w}(t)g(\varphi^{-1}(t))$. It is immediate to see that
  $\|S^n\| = \|\check{T}^n\|, n \in \mathds{N}$, and therefore $\sigma(S) \subseteq \mathds{T}$. It follows from Lemma 3.6 in~\cite{Ki} that there is a point $s \in \partial A$ such that
  
  \begin{equation}\label{eq6}
    |\check{w}_n(s)|\leq 1 \; \text{and} \; |\check{w}_n(\varphi^{-n}(s))| \geq 1, n \in \mathds{N},
  \end{equation}
where $\varphi^n, n \in \mathds{N}$ is the $n^{th}$ iteration of $\varphi$ and 
$\check{w}_n = \check{w}(\check{w}\circ \varphi) \ldots (\check{w} \circ \varphi^{n-1})$.

Let us assume like in the proof of Theorem~\ref{t1} that there exist an $\alpha \in \mathds{T}$ and an $r > 0$ such that $U(\alpha, r) \cap \sigma(T) = \emptyset$. It follows from the formula
 \begin{equation}\label{eq3} R(\lambda, \check{T}) = (\lambda I - \check{T})^{-1} =
 \left\{
   \begin{array}{ll}
     \sum \limits_{n=0}^\infty \frac{(\check{T})^n}{\lambda^{n+1}}, & \hbox{if $|\lambda| > 1$;} \\
     \sum \limits_{n=0}^\infty \lambda^n (\check{T})^{-(n+1)}, & \hbox{if $|\lambda| <1$,}
   \end{array}
 \right.
  \end{equation}
and from the equality~(\ref{eq4}) that for any $f \in A$ the function
 \begin{equation}\label{eq7} R(\lambda, \check{T})\check{f}(s) = 
 \left\{
   \begin{array}{ll}
     \sum \limits_{n=0}^\infty \frac{(\check{w}_n(s)\check{f}(\varphi^n(s)) }{\lambda^{n+1}}, & \hbox{if $|\lambda| > 1$;} \\
     \sum \limits_{n=0}^\infty \lambda^n \frac{\check{f}(\varphi^{-(n+1)}(s))}{\check{w}_{n+1}(\varphi^{-(n+1)}(s))}, & \hbox{if $|\lambda| <1$}
   \end{array}
 \right.
  \end{equation}
can be analytically extended on the disc $U(\alpha, r)$.

Let $g \in \check{A}$ and let $f_n \in A$ be such that $\|\check{f_n} - g\|_{\check{A}} \rightarrow 0$. Inequalities~(\ref{eq6}) guarantee that the family $R(\lambda, \check{T})f_n(s)$ is normal in $D(\alpha, r)$ an therefore the function $R(\lambda, \check{T})g(s)$ is analytic in $D(\alpha, r)$. Because $s$ is a peak point for $\check{A}$ and it is not $\varphi$-periodic, there is a sequence $g_m \in \check{A}$ such that $\|g_m\|=g_m(s) = 1$ and
$|g_m(\varphi^i(s))| < 1/m, 0 < |i| < m$. Applying again inequalities~(\ref{eq6}) we see that the family
$R(\lambda, \check{T})g_m(s)$ is normal in $D(\alpha, r)$ and therefore the function $1/\lambda$ continues analytically to $0$ through the arc $\mathds{T} \cap D(\alpha, r)$, a contradiction.
  \end{proof}

Next we intend to prove that the conditions in the statements of Theorems~\ref{t2} and~\ref{t3} guarantee that the spectrum of operator $T$ is connected, provided that $T$ is invertible. It follows from the next theorem.

\begin{theorem} \label{t4}
  Let $A$ be a unital semi-simple commutative Banach algebra, $U$ be an automorphism of $A$, and $w \in A$. Let
  $T = wU$ and let $\check{T}$ be the associated weighted composition operator on $\check{A}$. Assume that
\begin{enumerate} [(a)]
  \item $w$ is an invertible element of $A$.
  \item $\sigma(\check{T}) = \mathds{T} \subseteq \sigma(T)$.
\end{enumerate}
Then the set $\sigma(T)$ is connected.
\end{theorem}

\begin{proof}
  Assume to the contrary that $\sigma(T)$ is not connected. It follows from $(a)$ and $(b)$ that we can assume without loss of generality that there is a component $\sigma$ of $\sigma(T)$ and a real number $R$, $0 < R < 1$ such that $\sigma \subset D(0,R)$. Let $\Sigma \subset A$ be the spectral subspace of $T$ corresponding to $\sigma$ and let $f \in \Sigma, f \neq 0$. Then
\begin{equation*}
  \|\check{T}^n \check{f}\|_{\check{A}} \leq \|T^n f\| \leq CR^n,
\end{equation*}
in contradiction with the condition $\sigma(\check{T}) = \mathds{T}$.
\end{proof}

\begin{corollary} \label{c2}
  Let $A$ be a semi-simple commutative Banach algebra, $U$ be an automorphism of $A$ and $U^n \neq I, n \in \mathds{N}$. Then $\sigma(U)$ is connected.
\end{corollary}

\begin{remark} \label{r1}
  Corollary~\ref{c2} was stated (with an incorrect proof) in~\cite{Sc} and proved in~\cite{Hu}.
\end{remark}

Similarly to Theorem~\ref{t3} we can prove the following proposition.

\begin{proposition} \label{p1}
  Let $A$ be a unital semi-simple commutative Banach algebra, $U$ be an automorphism of $A$, and $w \in A$. Let
  $T = wU$ and let $\check{T}$ be the associated weighted composition operator on $\check{A}$. Assume that every point in $\partial A$ is a peak point for $\check{A}$ and there are a $\lambda \in \mathds{C}$ and a not $\varphi$-periodic point $s \in \partial A$ such that
\begin{equation}\label{eq8}
    |\check{w}_n(s)|\leq |\lambda|^n \; \text{and} \; |\check{w}_n(\varphi^{-n}(s))| \geq |\lambda|^n, n \in \mathds{N},
  \end{equation} 
Then $\lambda \mathds{T} \subseteq \sigma(T)$.
\end{proposition}

Combining Proposition~\ref{p1} with Lemma 3.6 and Theorem 3.29 in~\cite{Ki} we obtain the following:

\begin{theorem} \label{t5}
  Let $A$ be a unital semi-simple commutative Banach algebra, $U$ be an automorphism of $A$ and $w \in A$. Let $T = wU$ and let $\check{T}$ be the associated weighted composition operator on $\check{A}$. Assume that every point of $\partial A$ is a peak point for $\check{A}$. Let $\lambda \in \sigma(\check{T})$.

$(1)$ If $\lambda \mathds{T} \cap \sigma(T) = \emptyset$ then $\partial A$ is the union of three disjoint sets
$K_1, K_2, O$ with the properties
\begin{enumerate} [(i)]
  \item The sets $K_1$ and $K_2$ are closed and $|\check{w}| > 0$ on $K_2$.  
  \item $\rho(\check{T}, C(K_1)) < |\lambda|$ and $\rho(\check{T}^{-1}, C(K_2)) < |\lambda|$.
  \item If $t \in O$ then all accumulation points of the sequence $\varphi^n(s), n \in\mathds{N}$, lay in $K_2$ and all accumulation points of the sequence $\varphi^{-n}(s), n \in \mathds{N}$, lay in $K_1$.
\end{enumerate}

$(2)$ Assume additionally that the set of $\varphi$-periodic points is of first category in $\partial A$. If there is a $\lambda \in \sigma(\check{T})$ such that $\lambda \mathds{T} \cap \sigma(T) \neq \emptyset$, but
$\lambda\mathds{T} \not \subseteq \sigma(T)$ then there is a $\varphi$-periodic point $s \in \partial A$ of the smallest period $p$ such that $\check{w}_p(s) = \lambda^p$.
\end{theorem}

Recall that if $K$ is a compact Hausdorff space and $\varphi$ is a homeomorphism of $K$ onto itself then a point $k \in K$ is called $\varphi$-wandering if there is an open neighborhood $V$ of $k$ such that the sets $cl \, \varphi^n(V), n \in \mathds{Z}$, are pairwise disjoint.

\begin{corollary} \label{c3}
  Let $A$ be a unital semi-simple commutative Banach algebra, $U$ be an automorphism of $A$ and $w \in A$. Let
  $T = wU$ and let $\check{T}$ be the associated weighted composition operator on $\check{A}$. Assume that every point of $\partial A$ is a peak point for $\check{A}$. Assume that $\partial A$ contains no $\varphi$-wandering points. Then
 \begin{equation*}
  |\sigma(\check{T})| \subset |\sigma(T)|
\end{equation*}
Moreover, if we assume additionally that $\partial A$ contains no $\varphi$-periodic points then
\begin{equation*}
  \sigma(\check{T}) \subset \sigma(T).
\end{equation*}
\end{corollary}

We end this section with the following proposition.

\begin{proposition} \label{p2}
 Let $A$ be a unital semi-simple commutative Banach algebra, $U$ be an automorphism of $A$ and $w \in A$. Let
  $T = wU$ and let $\check{T}$ be the associated weighted composition operator on $\check{A}$. Then
  
  \begin{equation*}
    \sigma_r(\check{T}) = \sigma(\check{T}) \setminus \sigma_{a.p.}(\check{T}) \subset \sigma(T),
  \end{equation*}
  where $\sigma_{a.p.}(\check{T})$ is the approximate point spectrum of $\check{T}$.
\end{proposition}

\begin{proof}
  Let $\lambda \in \sigma_r(\check{T}), \lambda \neq 0$. By Theorem 3.29 from~\cite{Ki} $\partial A$ is the union of three disjoint sets $K_1, K_2, O$ with the properties
\begin{enumerate} [(a)]
  \item The sets $K_1$ and $K_2$ are closed and $|\check{w}| > 0$ on $K_2$.  
  \item $\rho(\check{T}, C(K_1)) < |\lambda|$ and $\rho(\check{T}^{-1}, C(K_2)) < |\lambda|$.
  \item If $t \in O$ then all accumulation points of the sequence $\varphi^n(s), n \in\mathds{N}$ lay in $K_1$ and all accumulation points of the sequence $\varphi^{-n}(s), n \in \mathds{N}$ lay in $K_2$.
\end{enumerate}
Let $s \in O$. Conditions (a) - (c) guarantee that there is a positive $\varepsilon$ such that for any $\gamma$, $|\lambda| - \varepsilon < |\gamma| < |\lambda| + \varepsilon$ the series
$ \mu_\gamma =\sum \limits_{n=0}^\infty \gamma^{-n} \check{w}_n(s)\delta_{\varphi^n(s)}$
$+ \sum \limits_{n=1}^\infty \gamma^n \frac{\delta_{\varphi^{-n}(s)}}{\check{w}_n(\varphi^{-n}(s))}$ converges by norm in $C^\prime(\partial A)$. It is immediate to see that $\mu_\gamma$ defines a continuous functional $F_\gamma$ on $A$, that $T^\prime F_\gamma = \gamma F_\gamma$ and that there are a $\gamma \in D(\lambda, \varepsilon)$ and an $f \in A$ such that $F_\gamma f \neq 0$. Because $\varepsilon$ is arbitrary small we have $\lambda \in \sigma(T)$.
\end{proof}

\section{Some special cases}

While, as we have seen in the previous section, for a weighted automorphism $T$ of a semi-simple unital commutative Banach algebra
$\sigma(T)  \cap \sigma(\check{T}) \neq \emptyset$, neither of the inclusions $\sigma(T) \subset \sigma(\check{T})$ or $\sigma(\check{T}) \subset \sigma(T)$ are in general true.

Moreover, to the best of our knowledge, the following problem (see e.g.~\cite{GK}) remains unsolved

\begin{problem} \label{pr2} Let $K$ be a compact connected subset of the complex plane. Assume that $\mathds{T} \subset K$. Are there a commutative semi-simple Banach algebra $A$ and an automorphism $U$ of $A$ such that $\sigma(U) = K$?
\end{problem}
Nevertheless, in some special cases we can prove the equality $\sigma(T) = \sigma(\check{T})$. We will need the following lemma proved in~\cite[Theorem 1]{Ki1}. 

\begin{lemma} \label{l1}
  Let $A$ be a semi-simple commutative Banach algebra, $U$ be an automorphism of $A$, $w \in A$, and $T = wU$. Assume that
  \begin{enumerate}
    \item There is a $C>0$ such that for any $f,g \in A$     
    \begin{equation}\label{eq9}
      \|fg\|_A \leq C(\|f\|_A \|\check{g}\|_{\check{A}} + \|\check{f}\|_{\check{A}} \|g\|_A ),
    \end{equation}
    \item $\sigma(U) \subset \mathds{T}$.
  \end{enumerate}
  Then, $\rho(T) = \rho(\check{T})$, where $\rho(T)$ is the spectral radius of $T$. 
\end{lemma}

In the sequel we will need a stronger version of inequality~(\ref{eq9}): there is a $C>0$ such that
\begin{equation}\label{eq10}
  \|f_1 \ldots f_p\| \leq C \sum \limits_{j=1}^p \|f_j\|_A \| \prod \limits_{k \neq j}\check{f_k}\|_{\check{A}}, p \in \mathds{N},
  f_j \in A, j = 1,\ldots,p.
\end{equation}

\begin{theorem} \label{t6}
  Let $A$ be a unital regular semi-simple commutative Banach algebra. Assume~(\ref{eq10}) and assume that every point of $\partial A$ is a peak point for the algebra $\check{A}$. Let $U$ be an automorphism of $A$ such that the set of all $\varphi$-periodic points is of first category in $\partial A$ and $\sigma(U) = \mathds{T}$. Let $w \in A^{-1}$ and $T = wU$. Then,
    \begin{equation*}
    \sigma(T) = \sigma(\check{T}).
  \end{equation*}
\end{theorem}

\begin{proof}
  (1) $\sigma(\check{T}) \subset \sigma(T)$. Let $S = \check{w}\check{U}^{-1}$. It follows from~\cite[Theorem 4.2]{Ki} and the equality $\mathfrak{M}_A = \partial A$ that $\sigma(\check{T}) = \sigma(S)$. If $\lambda \in \sigma_{a.p.}(S)$ then by Proposition~\ref{p1} $\lambda \in \sigma(T)$. If $\lambda \in \sigma_r(S)$ then by Theorem 3.29 in~\cite{Ki} $\partial A$ is the union of sets $K_1$, $K_2$, $O$ with properties (i) - (iii) from the statement of Theorem~\ref{t5}. Let $V$ be a nonempty open subset of $O$ such that the sets $cl(\varphi^n(V)), n \in \mathds{Z}$ are pairwise disjoint and let $f \in A$ be such that $f \neq 0$ and $supp f \subset V$. We claim that the series $g= \sum \limits_{n=-\infty}^\infty \lambda^{-n}T^nf$ converges by norm in $A$. This claim will follow as soon as we prove that there are $\gamma, 0 < \gamma < 1$, and $N \in \mathds{N}$ such that for any $n$ satisfying $|n|> N$ we have
   \begin{equation} \label{eq11}
    \|\lambda^{-n} T^nf\| \leq |n|\gamma^n \|w\| \|w^{-1}\| \|U^{-n}\|\|f\|.
  \end{equation}
  
  Let us consider the case $n \geq 0$. There is a $u \in A$ such that $\check{u} \equiv 0$ on some open neighborhood of the set $K_2$ and $\check{u} \equiv \check{w}$ on $\bigcup \limits_{n=0}^\infty \varphi^{-n}(V)$. Notice that $T^n f= u_nU^nf, n \geq 0$.
  It follows from properties (i) - (iii) and inequality~(\ref{eq10}) that there is an $N \in \mathds{N}$ such that for any $n \geq N$ inequality~(\ref{eq11}) is satisfied.
  
  Similarly, to prove~(\ref{eq11}) in case $n < 0$ we consider $v \in A$ such that $\check{v} \equiv 0$ on some open neighborhood of the set $K_1$ and $\check{v} \equiv \check{w}$ on $\bigcup \limits_{n=1}^\infty \varphi^n(V)$. Thus, $g \in A$ and clearly, $Tg = \lambda g$.
  
  (2) $\sigma(T) \subset \sigma(\check{T})$. Let $\lambda \in \mathds{C} \setminus \sigma(\check{T})$. By Lemma~\ref{l1} we can without loss of  generality assume that $1/\rho(\check{T}^{-1}) < |\lambda| < \rho(\check{T})$. It follows from Theorems 3.7 and 3.10 in~\cite{Ki} that $\mathfrak{M}_A = \partial A$ is the union of two disjoint closed $\varphi$-invariant subsets $K_1$ and $K_2$ such that $\rho(\check{T}, C(K_1)) < |\lambda|$ and $\rho(\check{T}^{-1}, C(K_2)) < |\lambda|$. It follows from Shilov's Idempotent Theorem and Lemma~\ref{l1} that $\lambda \not \in \sigma(T)$.
    \end{proof}
    
  The condition that every point in $\partial A$ is a peak point for
  $\check{A}$ is rather restrictive. In the next theorem we get rid of it but at the price of introducing some additional conditions.

  \begin{theorem} \label{t7}
   Let $A$ be a unital regular semi-simple commutative Banach algebra. 
   Assume the following "local" version of inequality~(\ref{eq10}), if
   $f_1, \ldots, f_p \in A$ and $supp(f_1f_2\ldots f_p) \subset V$, where $V$ is an open subset of $\partial A$ then
   \begin{equation}\label{eq12}
  \|f_1 \ldots f_p\| \leq C \sum \limits_{j=1}^p \|f_j\|_A \| \prod \limits_{k \neq j}\check{f_k}\|_{C(clV)},
   \end{equation}
   where the constant $C$ does not depend on $V$, $p$, and $f_1, \ldots, f_p$. Assume additionally that if $f,g \in A$ and $supp(\check{f}) \cap supp(\check{g}) = \emptyset$ then
      \begin{equation}\label{eq13}
     \|f+g\|_A = \|f-g\|_A.
   \end{equation}
 Let $U$ be an automorphism of $A$ such that the set of all $\varphi$-periodic points is of first category in $\partial A$ and $\sigma(U) = \mathds{T}$. Let $w \in A^{-1}$ and $T = wU$. Then
    \begin{equation*}
    \sigma(T) = \sigma(\check{T}).
  \end{equation*}
    \end{theorem}
      
    \begin{proof}
  The inclusion $\sigma(T) \subset \sigma(\check{T})$ follows from the reasoning in the proof of Theorem~\ref{t6}. By Proposition~\ref{p2} it suffices to prove that if $\lambda \in \sigma_{a.p.}(\check{T})$ then $\lambda \in \sigma(T)$. Without loss of generality we can assume that $\lambda = 1$. It follows from Lemma 3.6 in~\cite{Ki} and the condition that the set of $\varphi$-periodic points is of first category in $\partial A$ that for every $n \in \mathds{N}$ we can find an open subset $V_n$ of $\partial A$ with the properties
    \begin{equation} \label{eq14}
    \varphi^i(cl V_n) \cap \varphi^j(cl V_n) = \emptyset, -n-1 \leq i < j \leq n+1.
  \end{equation}
  and
    \begin{equation} \label{eq15}
    |w_i(t)| \geq 1/2, |w_i(\varphi^{-i}(t)| \leq 2, 0 \leq i \leq n+1, t \in V_n.
  \end{equation}
  Let $f_n \in A$ be such that $\|f_n\|_A = 1$ and $supp \check{f_n} \subset V_n$. Let
  
  \begin{equation}\label{eq16}
    g_n = \sum \limits_{j=-n-1}^n (1-\varepsilon_n)^{|j|}T^jf_n,
  \end{equation}
  where the positive numbers $\varepsilon_n, n \in \mathds{N}$, will be chosen later. Notice that in virtue of~(\ref{eq13}) $\|g_n\| \geq 1/2$. From~(\ref{eq16}) we get
  
  \begin{equation}\label{eq17}
  \begin{split}
   & g_n - Tg_n = (1-\varepsilon_n)^{-n-1}T^{-n-1}f_n \\
    & + \varepsilon_n \sum \limits_{j=-n}^0 T^jf_n - \varepsilon_n \sum \limits_{j=1}^n T^jf_n + (1-\varepsilon_n)^n T^{n+1}
    \end{split}
  \end{equation}
  It follows from~(\ref{eq12}), (\ref{eq13}), ~(\ref{eq14}), and~(\ref{eq17}) that
  
  \begin{equation}\label{eq18}
  \begin{split}
    & \|g_n - Tg_n\| \leq \varepsilon_n \|g_n\| \\ +[(1-\varepsilon_n)^{n+1} 
    & + \varepsilon_n]\|T^{-n-1}f_n\| +(1-\varepsilon_n)^n\|T^nf_n\|.
    \end{split}
  \end{equation}
In virtue of~\ref{eq12}) and~(\ref{eq15} we have
 
\begin{equation}\label{eq19}
\begin{split}
 & \|T^nf_n\|\leq n\|w\|\|w^{-1}\|\|U^n\| \; \text{and} \\
& \|T^{-n-1}f_n\|\leq (n+1)\|w\|\|w^{-1}\|\|U^{-n-1}\|.
\end{split}
\end{equation}
  Next we will define the numbers $\varepsilon_n$. Notice that 
$\|U^{|n|}\| = e^{\delta_n n}$, where $\limsup \delta_n = 0$. Let
  
  \begin{equation}\label{eq20}
   \varepsilon_n = \left\{
      \begin{array}{ll}
       2\delta_n , & \hbox{if $\limsup n\delta_n = \infty$;} \\
       1/\sqrt{n} , & \hbox{otherwise.}
      \end{array}
    \right.
  \end{equation}
It follows from~(\ref{eq18}), ~(\ref{eq19}), and~(\ref{eq20}) that $\|g_n - Tg_n\| \rightarrow 0$. 
    \end{proof}

\begin{remark} \label{r2}
  It follows from the proof of Theorem~\ref{t7} that 
$\sigma_{a.p.}(\check{T}) \subset \sigma_{a.p.}(T)$. Moreover, if we assume additionally that $\partial A$ is metrizable and contains no isolated points, then
$\sigma_{a.p.}(\check{T}) \subset \sigma_{usf}(T)$, where $\sigma_{usf}(T)$ denotes the upper semi-Fredholm spectrum of $T$.
\end{remark}

\section{Examples}

\begin{example} \label{ex1}
  Let $K$ be a compact metric space with the metric $d$ and $\varphi$ be an isometry of $K$ onto itself. Let $0 < \alpha \leq 1$ and
\begin{equation}\label{eq21}
\begin{split}
  & A = Lip_\alpha(K) = \{f \in C(K): \|f\| = \|f\|_{C(K)} \\
& +\sup \limits_{x,y \in K, x \neq y} \frac{|f(x)-f(y)|}{[d(x,y)]^\alpha} < \infty\}.
\end{split}
\end{equation}
 Assume that the set of all $\varphi$-periodic points is of first category in $K$. Let $w$ be an invertible element of the Banach algebra
$A=Lip_\alpha(K)$ and $Tf(k)=w(k)f(\varphi(k)),f \in A, k \in K$. By Theorem~\ref{t6} $\sigma(T,A) = \sigma(T,C(K))$. In particular, if $K$ is connected, then $\sigma(T,A)$ is an annulus or a circle centered at $0$.
\end{example}

Example~\ref{ex1} can be generalized as follows

\begin{example} \label{ex2}
 Let $K$ be a compact metric space with the metric $d$ and $\varphi$ be a continuous map of $K$ onto itself. Assume that 
$K = \bigcup \limits_{j=1}^n K_n$ where $K_j$ are clopen $\varphi$-invariant subsets of $K$ and that the restriction of $\varphi$ on each $K_i$ is an isometry of $K_i$ onto itself. Assume that the set of all $\varphi$-periodic points is of first category in $K$. Let $w$ be an invertible element of the Banach algebra
$A=Lip_\alpha(K)$ and $Tf(k)=w(k)f(\varphi(k)),f \in A, k \in K$. Then
$\sigma(T,A) = \sigma(T,C(K))$.
\end{example}
\begin{proof}
  Let $0 < \varepsilon < \min \limits_{1 \leq i < j \leq n} d(K_i,K_j)$. By~\cite[Theorem 2.1.6, p.101]{CMN} the norm on $A$ defined as
\begin{equation}\label{eq22}
    \|f\| = \|f\|_{C(K)}  +\sup \limits_{x,y \in K, 0 < d(x,y) \leq \varepsilon} \frac{|f(x)-f(y)|}{[d(x,y)]^\alpha}
\end{equation}
is equivalent to the norm~(\ref{eq21}). It remains to apply Theorem~\ref{t6}.
\end{proof}

\begin{example} \label{ex3}
 Let $A$ be the Banach algebra of all complex-valued continuously differentiable functions on the unit disc $\mathds{D}$ with the norm
\begin{equation*}
  \|f\| = \|f\|_\infty + \|\partial f / \partial x\|_\infty +\|\partial f / \partial y\|_\infty.
\end{equation*}
Let $\varphi$ be a non-periodic elliptic Möbius transformation of $\mathds{D}$, $Uf = f \circ \varphi, f \in A$, $w \in A^{-1}$, and 
$T=wU$. Because the automorphism $U$ is similar to the automorphism generated by a non periodic rotation of $\mathds{D}$ there is a $C > 0$ such that $\|U^n\| \leq C, n \in \mathds{Z}$, and therefore by Theorem~\ref{t6}
$\sigma(T,A) = \sigma(T,C(\mathds{D}))$. In particular, $\sigma(T,A)$ is the annulus with the radii
\begin{equation*}
\begin{split}
 & r= \min \limits_{\mu \in M_\varphi} \int \ln |w| d\mu, \\
&  R= \max \limits_{\mu \in M_\varphi} \int \ln |w| d\mu,
\end{split}
\end{equation*}
where $M_\varphi$ is the set of all $\varphi$-invariant probability Radon measures on $\mathds{D}.$ 
\end{example}

Example~\ref{ex3} can be generalized as follows

\begin{example} \label{ex4}
  Let $u$ be an orthogonal map of $\mathds{R}^n$ onto itself such that $u^n \neq I, n \in \mathds{N}$. Let $K$ be a compact subset of $\mathds{R}^n$ such that $uK = K$ and let $A$ be the closure of complex-valued polynomials in the norm $\|P\| = \|P\|_{C(K)} + \sum \limits_{k=1}^n \|\partial P/\partial x_k\|_{C(K)}$. Let $w \in A^{-1}$ and $(Tf)(x) = w(x)f(ux), f \in A, x \in K$. Then (see e.g.~\cite[p55]{Lo}) $\mathfrak{M}_A = \partial A = K$ and it is not difficult to see that every point of $K$ is a peak point of $A$. By Theorem~\ref{t6} $\sigma(T,A) = \sigma(T,C(K))$. In particular, $\sigma(T,A)$ is rotation invariant, and, if $K$ is connected, $\sigma(T)$ is an annulus or a circle centered at $0$.
\end{example}

\begin{example} \label{ex5}
  Let $A$ be the Wiener algebra of all the functions on $\mathds{T}$ with absolutely convergent Fourier series.
  Let $w \in A^{-1}$ and let $(Tf)(z) = w(z)f(\alpha z), f \in A, z \in \mathds{T}$, where $\alpha$ is not a root of unity. It follows from Theorem~\ref{t3} that $\sigma(T) \supseteq c\mathds{T}$, where
  $c = \exp \frac{1}{2\pi} \int \limits_0^{2\pi} \ln |w(e^{i\theta})|d\theta$. Moreover, it follows from Theorem~\ref{t6} that if $w$ satisfies the well known condition
\begin{equation*}
  \int \limits_0^1 \frac{\omega(\delta)}{\delta^{3/2}} < \infty,
\end{equation*}
where $\omega$ is the modulus of continuity of $w$, or if
 $w \in BV \cap Lip_\alpha(\mathds{T})$, where $0 <\alpha \leq 1$, then $\sigma(T) = c\mathds{T}$.
\end{example}

\begin{remark} \label{r3}
  We do not know whether the equality $\sigma(T) = c\mathds{T}$ is valid for an arbitrary weight $w \in A^{-1}$, or even if $\sigma(T)$ is rotation invariant.
\end{remark}

\end{document}